\documentclass[a4paper,11pt]{article}

\usepackage[margin=1in]{geometry}
\usepackage{setspace}
\usepackage{amsmath}
\usepackage{amssymb}
\usepackage{amsthm}
\usepackage{graphicx}
\usepackage{float}
\usepackage{caption}
\usepackage{listings}
\usepackage{tikz}
\usepackage{enumitem}

\newtheorem{theorem}{Theorem}[section]
\newtheorem{proposition}[theorem]{Proposition}
\newtheorem{corollary}[theorem]{Corollary}

\newtheorem{lemma}[theorem]{Lemma}
\newtheorem{definition}[theorem]{Definition}

\newtheorem{remark}[theorem]{Remark}

\numberwithin{figure}{section}

\providecommand{\keywords}[1]{\textbf{\textit{Keywords:}} #1}

\usepackage{hyperref}
\hypersetup{
    pdftitle={Hausdorff Dimension of Union of Lines Covering a Curve: Applications to Mathematical Physics},
    pdfauthor={Hanwen Liu}
}

\title{\texorpdfstring{\textbf{Hausdorff Dimension of Union of Lines Covering a Curve: Applications to Mathematical Physics}}{Hausdorff Dimension of Union of Lines Covering a Curve: Applications to Mathematical Physics}}
\author{Hanwen Liu}
\date{}

\begin{document}
\maketitle

\begin{abstract}
We prove that for any nonlinear $f \in C^{1,\alpha}([0,1])$, the union of lines covering its graph has a Hausdorff dimension of at least $1+\alpha$, and this dimension bound is sharp. We then apply these geometric results to mathematical physics, proving that spacetime observability sets for conservation laws with $\alpha$-H\"older initial wave speeds possess a dimension of at least $\alpha$. Finally, we prove that if an absolutely integrable vector field $v$ on the boundary of a polyhedron exhibits a strictly positive total flux, then the union of the line field spanned by $v$ possesses a Hausdorff dimension of 3.
\end{abstract} 

\begin{center}
\end{center}

\tableofcontents
\onehalfspacing
\raggedbottom

\section{Introduction and Background}

The study of how families of lines intersect, cover, or are constrained by sets in Euclidean space has a rich history in geometric measure theory. This area of research is intimately connected to the classical Kakeya problem, which asks for the minimal dimension of a set containing a unit line segment in every direction \cite{Davies1971, KatzTao2002, KatzZahl2019}. In 1952, R. O. Davies \cite{Davies1952} proved a foundational result regarding covers: any subset $A$ of the plane can be covered by a collection of lines such that the union of the lines has the exact same Lebesgue measure as $A$. 

However, when one shifts the perspective from Lebesgue measure to Hausdorff dimension, the geometry of line coverings becomes highly rigid and dependent on the structural smoothness of the object being covered. A natural variant of this problem asks: how small can the Hausdorff dimension of a set be if it contains a line passing through every point of a given planar curve? Relying on projection theorems, it is a well-known consequence that the dimension of a union of lines is deeply tied to the dimension of its parameterizing set. Recently, Venieri \cite{Venieri2017} proved that if a curve is Lipschitz, any union of transversally intersecting lines covering the curve must have a Hausdorff dimension of at least 2.

\begin{figure}[ht]
\centering
\includegraphics[width=0.75\textwidth]{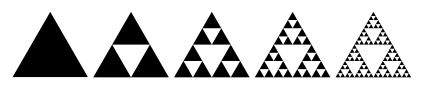}
\caption{The Sierpiński triangle of Hausdorff dimension $\alpha=\log(3)/\log(2)\approx1.585<2$. Throughout this article, we denote the Hausdorff dimension of a subset $E$ of the Euclidean space $\mathbb{R}^d$ by $\dim_\mathcal{H}(E)$.}
\end{figure}

When the covering family is permitted to include tangent lines, a striking dichotomy emerges based on the differentiability of the curve. Cumberbatch et al. recently demonstrated in \cite{Cumberbatch2023}  that if a nonlinear curve is twice-differentiable, the union of its covering lines must still possess a Hausdorff dimension of exactly 2. Astonishingly, they proved that this rigidity breaks down at lower regularities: it is possible to construct a continuously differentiable strictly convex curve that can be perfectly covered by a family of lines of which union has dimension $1$. 

We make six main contributions to this line of inquiry. First, in Section 2, we significantly strengthen a related result by Cumberbatch et al. regarding the tangent lines of differentiable functions. We prove that the union of tangent lines of a nonlinear differentiable function does not merely possess a Hausdorff dimension of 2, but must in fact contain an open disk.

Second, in Section 3, we apply these rigorous geometric constraints to geometric optics. We demonstrate that light rays mathematically tangent to a generic differentiable caustic cannot remain perfectly collimated or converge to a singular focal point, but are instead guaranteed to sweep out a macroscopic two-dimensional illuminated region.

Third, in Section 4, we resolve the mystery of the geometric rigidity in the intermediate spaces between $C^1$ and $C^2$. By studying curves parameterized by functions in the intermediate H\"older space $C^{1,\alpha}([0,1])$, we establish a sharp dimensional bound that perfectly bridges the gap. Our main theorem demonstrates that the minimal Hausdorff dimension of a union of lines covering a $C^{1,\alpha}$ nonlinear curve is exactly $1+\alpha$. 

To achieve this, we rely on the interaction between transversality, Lipschitz mappings, and the following foundational projection theorem by Falconer and Mattila:

\begin{lemma}[Falconer-Mattila, 2016 \cite{Falconer2016}]\label{Falconer}
Let $\Lambda$ be a measurable subset of $\mathbb{R}^2$, and $\pi\colon\mathbb{R}^2\rightarrow\mathbb{R}^1$ the projection onto the first coordinate. If $\dim_\mathcal{H}(\pi(\Lambda))\geq1$, then the Hausdorff dimension of $$\bigcup_{(a,b)\in \Lambda}\{(x,y)\in\mathbb{R}^2\mid y=ax+b\}$$ is exactly equal to $2$.
\end{lemma}

Fourth, in Section 5, we translate our geometric results into the realm of nonlinear partial differential equations, providing a novel dimension bound for the observability of characteristic lines in general first-order scalar conservation laws. 

Fifth, in Section 6, we focus on the measure theory of functions of bounded variation. Utilizing the Lebesgue decomposition theorem and the area formula, we prove that for a nonlinear $C^1$ function $f$ of which derivative is of bounded variation, if the union of some family of lines that cover $\operatorname{graph}(f)$ has Hausdorff dimension less than $2$, then the distributional derivative $f''$ is a singular measure.

Finally, in Section 7, as a corollary of Wang and Zahl's recent work in Kakeya problem, we prove that any absolutely integrable vector field defined on the boundary of a polyhedron that yields a strictly positive total flux forces the corresponding line field to sweep out a set of full Hausdorff dimension.

The appendices manifest attempts to generalize our established results to high dimensions.

\section{Graphs of Differentiable Functions}

We begin by examining the local geometric behavior of tangent families. In Theorem 2.4 of their recent work, Cumberbatch, Keleti, and Zhang \cite{Cumberbatch2023} proved that the union of the tangent lines of a differentiable function with a non-constant derivative has a Hausdorff dimension of exactly 2. 

In this section, we show that this geometric structure is far more rigid. We establish that the union of these tangent lines cannot be a highly porous or fractal set of dimension 2; rather, the nonlinearity of the curve forces the tangent lines to sweep out positive area, guaranteeing that their union contains an open disk, or equivalently, a set with non-empty interior. We first establish a crucial connectivity property.

\begin{lemma}\label{differentiable1}
    Let $f\colon [0, 1]  \rightarrow \mathbb{R}$ be a differentiable function. For every $0<a\leq1$ the following properties hold:
    \begin{enumerate}
        \item[1.] $ \{f(x)+f^{\prime}(x)(a-x) \mid 0 \leq x \leq a \}$ is a connected subset of $\mathbb{R}$;
        \item[2.] $ \{f(x)+f^{\prime}(x)(a-x) \mid 0 \leq x \leq a \}$ is a singleton if and only if $f^{\prime}([0, a])$ is.
    \end{enumerate}
\end{lemma}
\begin{proof}
    Define the differentiable function $F\colon [0, a]  \rightarrow \mathbb{R}$ by $$F(x):=f(x)(a-x)+2 \displaystyle\int_0^x f(t) d t$$ then $$F^{\prime}(x)=f^{\prime}(x)(a-x)-f(x)+2 f(x)=f(x)+f^{\prime}(x)(a-x)$$ for all $x \in[0, a]$. By Darboux's theorem, we have that $F^{\prime}([0, a])$ is connected.

    Assume that $F^{\prime}$ is constant. Then there exists $b \in \mathbb{R}$ such that $$f(x)+f^{\prime}(x)(a-x) \equiv b$$ for all $x \in[ 0, a]$. Solving the first order linear ordinary differential equation $$(a-x) y^{\prime}+y-b=0$$ by separation of variables, we obtain the general solution $y=k(x-a)+b$ where $k \in \mathbb{R}$ is a constant of integration. Therefore $f^{\prime}(x) \equiv k$ for all $x \in[0, a]$.

    Conversely, assume that there exist $k, c \in \mathbb{R}$ such that $f(x)=k x+c$ for all $x \in[0, a]$ then $$f(x)+f^{\prime}(x)(a-x)=k x+c+k(a-x)\equiv a k+c$$ for all $x \in[0, a]$.
\end{proof}

Building upon this connectivity, we can now prove that the union of tangent lines generates a subset of the plane with non-empty interior.

\begin{lemma}\label{differentiable2}
    Let $f\colon [0, 1] \longrightarrow \mathbb{R}$ be a nonlinear differentiable function. Then the interior of the subset $$ \{ (y, f(x)+f^{\prime}(x)(y-x)) \mid 0 \leq x \leq y \leq 1 \}$$ of $\mathbb{R}^2$ is nonempty.
\end{lemma}
\begin{proof}
    Since $f^{\prime}$ is not constant, by Lemma \ref{differentiable1} there exist distinct $0\leq x_1, x_2 <1$ such that $$f (x_1 )+f^{\prime} (x_1 ) (1-x_1 )<f (x_2 )+f^{\prime} (x_2 ) (1-x_2 ).$$ For each $i \in\{1,2\}$, define $L_i\colon \mathbb{R}\rightarrow\mathbb{R}$ by $$L_i(t):=f (x_i )+f^{\prime} (x_i ) (t-x_i )$$ and denote $$\varepsilon:=\dfrac{1}{4} |L_2(1)-L_1(1) |>0.$$ Since $L_1$ and $L_2$ are continuous, there exists $\delta>0$ satisfying $\delta+\max  \{x_1, x_2 \}<1$ such that $$\sum_{i=1}^2 |L_i(t)-L_i(1) |<\varepsilon$$ for all $t \in[1-\delta, 1]$. If we write $$b:=\frac{L_1(1)+L_2(1)}{2}$$ then $L_1(t)<b-\varepsilon$ and $L_2(t)>b+\varepsilon$ for all $t \in[1-\delta, 1]$. Applying Lemma \ref{differentiable1} then yields $$[b-\varepsilon, b+\varepsilon] \subseteq \{f(x)+f^{\prime}(x)(t-x) \mid 0 \leq x \leq t \}$$ for all $t \in[1-\delta, 1]$, and hence $$[1-\delta, 1] \times[b-\varepsilon, b+\varepsilon] \subseteq \{ (y, f(x)+f^{\prime}(x)(y-x) ) \mid 0 \leq x \leq y, 0 \leq y \leq 1 \}$$ as desired.
\end{proof}

\begin{corollary}\label{differentiable3}
    Let $f\colon [0, 1] \longrightarrow \mathbb{R}$ be a nonlinear differentiable function. Then the union of tangent lines of the graph of $f$ contains an open disk, that is, the interior of the subset $$ \bigcup_{a\in[0,1]}\{ (a+t, f(a)+f^{\prime}(a)t) :t\in\mathbb{R} \}$$ of $\mathbb{R}^2$ is nonempty.
\end{corollary}
\begin{proof}
    This is an immediate consequence of Lemma~\ref{differentiable2}.
\end{proof}

To geometrically visualize the stark contrast established by these results, consider the behavior of tangent lines for linear versus non-linear functions. As illustrated in Figure \ref{tangent}, the tangent lines of a linear function trivially collapse into a single one-dimensional line. Conversely, the tangent lines of a strictly convex function continuously change direction, causing them to fan out and sweep across a two-dimensional region of the plane. This fanning effect guarantees that their union contains an open disk, as proven in Corollary \ref{differentiable3}.

\begin{figure}[H]
  \centering
  \begin{tikzpicture}[x=0.75pt,y=0.75pt,yscale=-1,xscale=1]

\draw [line width=2.25]    (329,249.6) .. controls (349.99,233.86) and (379.51,228.03) .. (404.13,229.66) .. controls (426.44,231.13) and (444.72,238.72) .. (449,250.6) ;
\draw [color={rgb, 255:red, 245; green, 166; blue, 35 }  ,draw opacity=1 ][line width=1.5]    (329,249.6) -- (374.55,217.34) ;
\draw [shift={(377,215.6)}, rotate = 144.69] [color={rgb, 255:red, 245; green, 166; blue, 35 }  ,draw opacity=1 ][line width=1.5]    (14.21,-4.28) .. controls (9.04,-1.82) and (4.3,-0.39) .. (0,0) .. controls (4.3,0.39) and (9.04,1.82) .. (14.21,4.28)   ;
\draw [color={rgb, 255:red, 245; green, 166; blue, 35 }  ,draw opacity=1 ][line width=1.5]    (368,232.6) -- (420.12,217.44) ;
\draw [shift={(423,216.6)}, rotate = 163.78] [color={rgb, 255:red, 245; green, 166; blue, 35 }  ,draw opacity=1 ][line width=1.5]    (14.21,-4.28) .. controls (9.04,-1.82) and (4.3,-0.39) .. (0,0) .. controls (4.3,0.39) and (9.04,1.82) .. (14.21,4.28)   ;
\draw [color={rgb, 255:red, 245; green, 166; blue, 35 }  ,draw opacity=1 ][line width=1.5]    (419,230.6) -- (462.13,243.73) ;
\draw [shift={(465,244.6)}, rotate = 196.93] [color={rgb, 255:red, 245; green, 166; blue, 35 }  ,draw opacity=1 ][line width=1.5]    (14.21,-4.28) .. controls (9.04,-1.82) and (4.3,-0.39) .. (0,0) .. controls (4.3,0.39) and (9.04,1.82) .. (14.21,4.28)   ;
\draw [line width=2.25]    (129,252) -- (248,221.6) ;
\draw [color={rgb, 255:red, 245; green, 166; blue, 35 }  ,draw opacity=1 ][line width=1.5]    (129,252) -- (176.09,240.32) ;
\draw [shift={(179,239.6)}, rotate = 166.07] [color={rgb, 255:red, 245; green, 166; blue, 35 }  ,draw opacity=1 ][line width=1.5]    (14.21,-4.28) .. controls (9.04,-1.82) and (4.3,-0.39) .. (0,0) .. controls (4.3,0.39) and (9.04,1.82) .. (14.21,4.28)   ;
\draw [color={rgb, 255:red, 245; green, 166; blue, 35 }  ,draw opacity=1 ][line width=1.5]    (154,245.8) -- (207.09,232.34) ;
\draw [shift={(210,231.6)}, rotate = 165.77] [color={rgb, 255:red, 245; green, 166; blue, 35 }  ,draw opacity=1 ][line width=1.5]    (14.21,-4.28) .. controls (9.04,-1.82) and (4.3,-0.39) .. (0,0) .. controls (4.3,0.39) and (9.04,1.82) .. (14.21,4.28)   ;
\draw [color={rgb, 255:red, 245; green, 166; blue, 35 }  ,draw opacity=1 ][line width=1.5]    (188.5,236.8) -- (245.58,221.98) -- (271.1,215.36) ;
\draw [shift={(274,214.6)}, rotate = 165.45] [color={rgb, 255:red, 245; green, 166; blue, 35 }  ,draw opacity=1 ][line width=1.5]    (14.21,-4.28) .. controls (9.04,-1.82) and (4.3,-0.39) .. (0,0) .. controls (4.3,0.39) and (9.04,1.82) .. (14.21,4.28)   ;
  \end{tikzpicture}
  \caption{An example in $\mathbb{R}^2$: The linear function on the left of which union of tangent lines has Hausdorff dimension 1, comparing with the strictly convex function on the right of which union of tangent lines has Hausdorff dimension 2.}
  \label{tangent}
\end{figure}
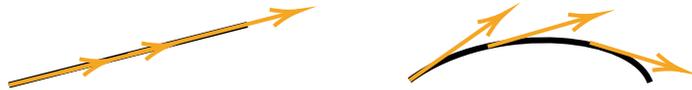

\section{Applications to Geometric Optics}

The geometric measure theory results established in the previous section possess profound and far-reaching implications for the fundamental behavior of light in geometric optics and the study of optical aberrations. The intrinsic connection lies in the physical propagation of light rays, which naturally translates the dynamics of wave propagation into intricate 1-parameter families of straight lines spanning the physical domain. By interpreting our geometric constraints in this context, we can derive the existence of illuminated regions.

\begin{figure}[ht]
\centering
\includegraphics[width=0.5\textwidth]{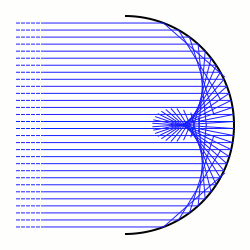}
\caption{Geometric visualization of ray intersection forming a caustic. A nonlinear wavefront set forces emanating normal lines to overlap, creating a possibly singular 1D envelope.}
\label{fig:caustic}
\end{figure}

In geometric optics, light waves propagate along straight rays that are strictly orthogonal to the propagating wavefront. By classical differential geometry from the 18th century, the normal lines of a given curve are exactly the tangent lines to its evolute. In an optical system, this evolute is precisely the caustic.

Formally applying our earlier results to characteristic families of light rays of a geometric optical system, we arrive at the following theorem:

\begin{theorem}
If the caustic set of a 2D geometric optical system contains a nonlinear differentiable arc, then the light rays of this system must illuminate an open region somewhere in the 2D plane.    
\end{theorem}

\begin{remark}
Corollary \ref{differentiable3} also applies to Clairaut's equation
$$
y=x\frac{dy}{dx}+f\left(\frac{dy}{dx}\right)
$$
for strictly convex $f\in C([0,1])$. We omit the discussion here.
\end{remark}

\section{H\"older and Sobolev Spaces}

Having established the robust geometric behavior of tangent families, we now arrive at the measure theoretic core of our study. We introduce our main theorem, which determines the absolute lower bound for the Hausdorff dimension of any family of lines covering the graph of a $C^{1,\alpha}$ function. To ensure this result is sufficiently adaptable for physical applications, where phenomena like shocks may occur on sets of measure zero, we frame the theorem such that the lines need only cover the graph over a subset of sufficiently large dimension.

Before we can state and prove our key theorem, we shall review the standard notion of analytic sets, also known as Suslin sets in the descriptive set theory literature, and then carefully recall the famous uniformisation theorem of Jankov and von Neumann.

As per usual, a subset of the Euclidean space $\mathbb{R}^d$ is said to be an analytic set, if it is a continuous image of a Borel subset of $\mathbb{R}^d$. It is well-known from point set topology that analytic sets form a strictly larger class than that of Borel sets. 

Notice that the class of analytic sets is closed under countable unions and intersections, as well as under continuous images or inverse images. Another core property of analytic sets is that they are always Lebesgue measurable. 

For our purposes, the most important hallmark of analytic sets is revealed by the following classical theorem.

\begin{theorem}[Jankov-von Neumann, 1949 \cite{Kechris1995}]\label{Jankov}
Let $A$ be an analytic set in $\mathbb{R}^2$ such that its projection to the $x$-axis is $B$. Then, there is a Lebesgue measurable function $g:B\rightarrow\mathbb{R}$ such that the graph of $g$ is a
subset of $A$.
\end{theorem}

Now, we are fully prepared to introduce our main theorem concerning Hölder functions in this section:

\begin{theorem}\label{Hölder}
Let $0<\alpha\leq1$, and let $I$ be a subset of $[0,1]$ such that $\dim_\mathcal{H}([0,1] \setminus I) < \alpha$. Let $f\in C^{1,\alpha}([0,1])$ and let $\Lambda$ be a Borel subset of $\mathbb{R}^2$. Suppose that the set $$E:=\bigcup_{(a,b)\in \Lambda}\{(x,y)\in\mathbb{R}^2\mid y=ax+b\}$$ contains $\{(x,f(x)):x\in I\}$. Then $\dim_\mathcal{H}(E)\geq \alpha+1$.
\end{theorem}
\begin{proof}
By virtue of Theorem \ref{Jankov}, without loss of generality, we may assume that there exists a measurable function $m\colon I \longrightarrow \mathbb{R}$ such that 
$$
\Lambda=\{(m(x),f(x)-m(x)x):x\in I\}.
$$
Define $A:=\{x\in I:m(x)=f'(x)\}$ and $B:=I\backslash A$. Denote $$
E_A:=\bigcup_{t\in A}\{(x,y)\in\mathbb{R}^2\mid y=m(t)(x-t)+f(t)\}
$$
and
$$
E_B:=\bigcup_{t\in B}\{(x,y)\in\mathbb{R}^2\mid y=m(t)(x-t)+f(t)\}.
$$

Now, assume the contrary that $\dim_\mathcal{H}(E)<\alpha+1$. Since $E_A\subseteq E$, we have $$\dim_\mathcal{H}(E_A)<\alpha+1\leq2.$$ By Lemma \ref{Falconer}, the Hausdorff dimension of $f'(A)$ is strictly less than $1$. Since the function $f\in C^1([0,1])$ is nonlinear, by the intermediate value theorem, there exist real numbers $a<b$ such that $f'([0,1])=[a,b]$, and hence $\dim_\mathcal{H}(f'([0,1]))=1$. Since $f'$ is $\alpha$-H\"older continuous and $\dim_\mathcal{H}([0,1] \setminus I) < \alpha$, we have $$\dim_\mathcal{H}(f'([0,1] \setminus I)) \leq \frac{1}{\alpha}\dim_\mathcal{H}([0,1] \setminus I) < 1.$$ Because $f'([0,1]) \subseteq f'(A) \cup f'(B) \cup f'([0,1] \setminus I)$, the Hausdorff dimension of $f'(B)$ must be exactly $1$. Since $f'$ is $\alpha$-H\"older continuous, we have 
$$
\frac{1}{\alpha}\dim_\mathcal{H}(B)\geq \dim_\mathcal{H}(f'(B))=1,
$$
that is, the Hausdorff dimension of $B$ is at least $\alpha$.

Since coordinate projections are Lipschitz, the Hausdorff dimension of $$\Gamma:=\{(x,m(x))\mid x\in B\}$$ is also at least $\alpha$. Since the continuously differentiable mapping 
$$
\Phi\colon [0,1]\times\mathbb{R}\rightarrow\mathbb{R}^2, \quad (x,y)\mapsto(y,f(x)-xy)
$$
has Jacobian determinant $y-f'(x)$, its restriction on $\{(x,y)\in[0,1]\times\mathbb{R}:y\neq f'(x)\}$ is bi-Lipschitz. Notice that $\Gamma$ is a subset of $\{(x,y)\in[0,1]\times\mathbb{R}:y\neq f'(x)\}$, and that $\Phi(\Gamma)$ is a subset of $\Lambda$. We therefore obtain that 
$$
\dim_\mathcal{H}(\Lambda)\geq\dim_\mathcal{H}(\Phi(\Gamma))=\dim_\mathcal{H}(\Gamma)\geq\alpha.
$$

Now, for each $x\in[0,1]$, consider the vertical slice $E_x:=\{ax+b\mid (a,b)\in\Lambda\}$ of $E$ over $x$. By the strong Marstrand theorem for projections, we have that 
$$
\dim_\mathcal{H}(E_x)\geq\min\{1,\dim_\mathcal{H}(\Lambda)\}\geq\alpha
$$
for Lebesgue a.e. $x\in[0,1]$. But by Marstrand's slicing theorem, we also have 
$$
\dim_\mathcal{H}(E_x)\leq\dim_\mathcal{H}(E)-1
$$
for Lebesgue a.e. $x\in[0,1]$. Therefore, we finally conclude that 
$$
\dim_\mathcal{H}(E)\geq\dim_\mathcal{H}(E_x)+1\geq\alpha+1
$$
as required.
\end{proof}


A natural question arises: is the lower bound of $\alpha + 1$ merely an artifact of the proof technique, or is it a strict geometric necessity? 

In the following theorem, we demonstrate that this bound is indeed sharp. By utilizing the fractional Devil's staircase built over an $\alpha$-dimensional Cantor dust set, we explicitly construct a convex function and a covering family of lines that exactly achieves this minimal possible dimension.

\begin{theorem}\label{example}
Let $\alpha\in[0,1]$. There exist a convex function $f\in C^{1,\alpha}([0,1])$ and a Borel subset $\Lambda$ of $\mathbb{R}^2$, such that the set $$E:=\bigcup_{(a,b)\in \Lambda}\{(x,y)\in\mathbb{R}^2\mid y=ax+b\}$$ contains the graph of $f$ and $\dim_\mathcal{H}(E)=\alpha+1$.
\end{theorem}
\begin{proof}
The case of $\alpha=1$ is trivial. We therefore assume $0\leq\alpha<1$ from now on.

Let $C$ be a Cantor set of Hausdorff dimension $\alpha$, viewed as a Lebesgue null subset of $[0,1]$. Let $g\colon[0,1]\rightarrow\mathbb{R}$ be the Devil's staircase, that is, the cumulative probability distribution function of the $\alpha$-dimensional Hausdorff measure on $C$. Then, we have that $g$ is $\alpha$-H\"older continuous, and its restriction on $[0,1]\backslash C$ is locally constant. We define the convex function $f\in C^{1,\alpha}([0,1])$ by
$$
f(x)=\int_0^xg(t)dt.
$$
Also, define $A:=\{(2,f(x)-2x)\mid x\in C\}$ and $B:=\{(f'(x),f(x)-f'(x)x)\mid x\in [0,1]\backslash C\}$. Since $A$ is a continuous image of the compact set $C$, and $B$ is countable by construction, we have that $A\cup B$ is a Borel subset of $\mathbb{R}^2$.
Finally, we denote $$E_A:=\bigcup_{(a,b)\in A}\{(x,y)\in\mathbb{R}^2\mid y=ax+b\}$$ and $$E_B:=\bigcup_{(a,b)\in B}\{(x,y)\in\mathbb{R}^2\mid y=ax+b\}.$$ It suffices to prove that $\dim_\mathcal{H}(E_A\cup E_B)=\alpha+1$.

Since $E_B$ is a countable union of straight lines, we have that $\dim_\mathcal{H}(E_B)=1$. Consider the function $h\in C^{1,\alpha}([0,1])$ defined by $h(x):=f(x)-2x$. Since $h'=g-2$ and $0\leq g\leq1$, we have that $h$ is bi-Lipschitz onto its image. Therefore 
$$
\dim_\mathcal{H}(A)=\dim_\mathcal{H}(h(C))=\dim_\mathcal{H}(C)=\alpha.
$$
It is then readily seen that $\dim_\mathcal{H}(E_A)=1+\alpha$. We therefore conclude that $$\dim_\mathcal{H}(E_A\cup E_B)=\max\{1,1+\alpha\}=1+\alpha$$ as desired.
\end{proof}

\begin{figure}[ht]
\centering
\includegraphics[width=0.5\textwidth]{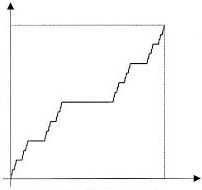} 
\caption{The fractional Devil's staircase, which serves as the derivative $f'(x)$ in our sharp dimensional construction, demonstrating constant intervals over the gaps of the Cantor set.}
\label{devils_staircase}
\end{figure}

To visually grasp the nature of the function $f$ constructed in Theorem \ref{example}, one considers its derivative, the fractional Devil's staircase. As depicted in Figure \ref{devils_staircase}, this function remains locally constant on the gaps of the Cantor set, mapping these intervals into a countable set of values. This geometric plateauing is precisely what forces the dimensional bounds of the tangent family to perfectly collapse to $1+\alpha$.

By combining our main theorem with standard Sobolev inequalities, we can readily extend these geometric bounds to functions defined in fractional Sobolev spaces, a setting heavily utilized in partial differential equations.

\begin{corollary}
Let $s,p\geq1$. Let $f\in W^{s,p}([0,1])$ be a nonlinear function, and let $\Lambda$ be a Borel subset of $\mathbb{R}^2$. Suppose that $(s-1)p>1$, and the set $$E:=\bigcup_{(a,b)\in \Lambda}\{(x,y)\in\mathbb{R}^2\mid y=ax+b\}$$ contains the graph of $f$. Then $\dim_\mathcal{H}(E)\geq \min\{s-1/p,2\}$.
\end{corollary}
\begin{proof}
By the Sobolev embedding theorem, we have $f\in C^{1,\alpha}([0,1])$ for $$\alpha:=(s-1)-\frac{1}{p}.$$ Now apply Theorem \ref{Hölder}.
\end{proof}

\begin{corollary}
Let $p\geq1$. Let $\Lambda$ be a Borel subset of $\mathbb{R}^2$, and $f\in W^{2,p}([0,1])$ a nonlinear function. Suppose that the set $$E:=\bigcup_{(a,b)\in \Lambda}\{(x,y)\in\mathbb{R}^2\mid y=ax+b\}$$ contains the graph of $f$. Then $\dim_\mathcal{H}(E)=2$.
\end{corollary}
\begin{proof}
Since $f\in W^{2,p}([0,1])$, we have that $f'$ is absolutely continuous, and hence maps Lebesgue null sets to Lebesgue null sets. The desired result then follows directly from the theorem on Luzin's property in \cite{Cumberbatch2023}.
\end{proof}

\section{Applications to Conservation Laws}

The geometric measure theory results established in the previous sections have profound and unexpected applications to the kinematic analysis of nonlinear partial differential equations. Specifically, the method of characteristics naturally generates 1-parameter families of straight lines in spacetime. By interpreting our geometric constraints in this physical context, we can derive rigorous dimensional bounds for the observability of shock-forming systems.

\begin{corollary}[Observability of Characteristics for Conservation Laws]\label{PDE}

Let $u=u(x,t)$ be the solution to the conservation law 
\begin{equation} \label{eq:Conservation}
\dot{u} + \partial_xF(u)  = 0
\end{equation}
subject to the initial condition $u(x,0) = f(x)$. The initial wave speed of equation (\ref{eq:Conservation}) is then $g(x) := F'(f(x))$. Suppose that $g \in C^{1,\alpha}([0,1])$ for some $0 < \alpha \le 1$ and that $g$ is nonlinear. If a Borel subset $\Lambda$ of the space-time plane intersects the characteristic lines of equation (\ref{eq:Conservation}) emanating from all initial positions in the spatial domain $[0,1]$, then $\dim_{\mathcal{H}}(\Lambda) \geq \alpha$.
\end{corollary}

\begin{proof}
By the method of characteristics, the characteristic line $L_s$ of equation (\ref{eq:Conservation}) emanating from the initial position $s \in [0,1]$ is defined by the equation
\begin{equation} \label{eq:char_line}
    x = g(s)t + s.
\end{equation}
By assumption, for every $s \in [0,1]$, there exists at least one point $(x(s), t(s)) \in L_s\cap\Lambda$. Define
\[ N := \{s \in [0,1] : L_s \cap \Lambda \subseteq [0,1]\times\{0\} \}. \]
For any $s \in N$, by construction, we have $(s, 0)\in\Lambda$. Therefore $\Lambda \subseteq N \times \{0\}$. Since the projection onto the $x$-axis is a Lipschitz mapping, if $\dim_{\mathcal{H}}(N) \ge \alpha$, then we immediately obtain
\[ \dim_{\mathcal{H}}(\Lambda) \ge \dim_{\mathcal{H}}(N \times \{0\}) = \dim_{\mathcal{H}}(N) \ge \alpha, \]
and the proof is complete.

We may therefore assume that $\dim_{\mathcal{H}}(N) < \alpha \le 1$. Because the Hausdorff dimension of $N$ is strictly less than 1, its 1-dimensional Lebesgue measure is zero. Let $A := [0,1] \setminus N$. It follows that $A$ has full Lebesgue measure in $[0,1]$. Also, for every $s \in A$, there exists a point $(x, t) \in L_s\cap\Lambda$ such that $t > 0$.

Define $\Lambda^+ := \Lambda \cap \{(x,t) \in \mathbb{R}^2 \mid t > 0\}$. Consider the mapping $$\Phi: \{(x,t) \in \mathbb{R}^2 \mid t > 0\} \to \mathbb{R}^2$$ defined by
\[ \Phi(x,t) =\frac{1}{t} ( -1, x). \]
Since $\Phi$ is real analytic, it is locally bi-Lipschitz. Therefore
\begin{equation} \label{eq:dim_tilde}
    \dim_{\mathcal{H}}(\Phi(\Lambda^+)) \le \dim_{\mathcal{H}}(\Lambda^+) \le \dim_{\mathcal{H}}(\Lambda).
\end{equation}

Next, we define
\[ E := \{(x, ax + b)\mid x\in\mathbb{R}, (a,b) \in \Phi(\Lambda^+)\} \]
and define $\varphi: \Phi(\Lambda^+) \times \mathbb{R} \to \mathbb{R}^2$ by $\varphi(a,b\,;x) := (x, ax+b)$. The image of $\varphi$ is exactly $E$. Since $\varphi$ is locally Lipschitz, we have
\begin{equation} \label{eq:dim_E_upper}
    \dim_{\mathcal{H}}(E) \le \dim_{\mathcal{H}}(\Phi(\Lambda^+) \times \mathbb{R}) = \dim_{\mathcal{H}}(\Phi(\Lambda^+)) + 1.
\end{equation}

Now, for any $s \in A$, there exists $(x,t) \in \Lambda^+$ satisfying $x = g(s)t + s$. Dividing by $t$ yields
\[ g(s) = -\frac{1}{t} s + \frac{x}{t}. \] 
Therefore, we have $E\supseteq\{(x,g(x)):x\in A\}$. Because $g \in C^{1,\alpha}([0,1])$ is nonlinear, by Theorem \ref{Hölder}, we conclude that
\begin{equation} \label{eq:dim_E_lower}
    \dim_{\mathcal{H}}(E) \ge 1 + \alpha.
\end{equation}

Inequalities \eqref{eq:dim_tilde},\eqref{eq:dim_E_upper},\eqref{eq:dim_E_lower} together finally yield
\[ \alpha+1 \le \dim_{\mathcal{H}}(E) \le \dim_{\mathcal{H}}(\Phi(\Lambda^+)) + 1 \le \dim_{\mathcal{H}}(\Lambda) + 1, \]
and thus 
$\dim_{\mathcal{H}}(\Lambda) \ge \alpha$.
\end{proof}

\begin{remark}[Physical Significance and Exact Controllability]
\hfil

\textnormal{The result established in Corollary \ref{PDE} carries an interesting and profound physical interpretation regarding the geometric control and observability of nonlinear wave propagation. In the context of conservation laws, characteristic lines represent the spacetime trajectories along which initial data propagates. In modern PDE control theory, if one wishes to deploy a set $\Lambda$ of sensors in spacetime to completely observe or reconstruct the initial state of the system, the sensor network must intersect the characteristic line emanating from every possible initial point.}

\textnormal{It is physically tempting to look for a trivial observability set $\Lambda$ of dimension zero. Indeed, if the initial wave speed function $g(x) = F'(f(x))$ is linear, the characteristic lines can perfectly focus into a single spacetime focal point, creating an instantaneous shock. Placing a single sensor at this exact focal point intersects every characteristic simultaneously, implying that full observability can be achieved with a discrete set of dimension 0.}

\begin{figure}[ht]
\centering
\includegraphics[width=0.7\textwidth]{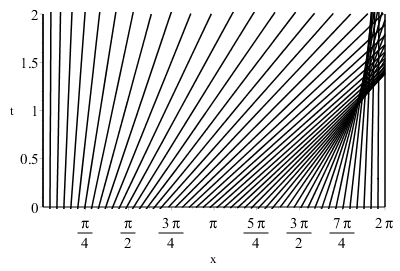} 
\caption{Characteristic lines in the space-time plane for Burger's equation. A linear initial wave speed yields perfect shock focusing, while a general non-linear $C^{1,\alpha}$ wave speed creates a smeared caustic requiring a higher-dimensional sensor.}
\label{shock_waves}
\end{figure}

\textnormal{Figure \ref{shock_waves} provides a geometric visualization of this physical phenomenon. For a linear initial wave speed, the characteristic lines converge perfectly to a single point, i.e. a perfect shock, allowing a 0-dimensional observability set. For a $C^{1,\alpha}$ wave speed, however, the non-linearity forces the characteristics to form a diffuse caustic, necessitating an observability network of dimension at least $\alpha$.}

\textnormal{However, our corollary rigorously proves that this perfect 0-dimensional focusing is entirely an artifact of a linear initial wave speed. When the physical interaction between the initial data $f(x)$ and the flux function $F(u)$ produces a wave speed $g(x)$ that is nonlinear and parameterized by an intermediate H\"older regularity $C^{1,\alpha}$, the varying slopes of the characteristics physically prevent them from perfectly aligning into a single focal point. Instead, the resulting shocks and caustics are necessarily smeared across spacetime. Consequently, any sensor network or control domain $\Lambda$ capable of intersecting all information pathways must maintain a minimal Hausdorff dimension of $\alpha$. This reveals a deep and necessary connection between the fractal dimension of spacetime observatories and the regularity of the physical waves.}
\end{remark}

\section{Functions of Bounded Variation}

In the preceding sections, we observed a striking geometric dichotomy: While $C^2$ curves rigidly force their covering lines to sweep out a 2-dimensional set, relaxing the regularity to $C^{1,\alpha}$ allows this dimension to drop. A natural question then arises: Within the context of differentiable curves, what is the exact boundary that dictates this dimensional collapse? 

To answer this, we transition from classical pointwise derivatives to the realm of distributional derivatives and measure theory. By assuming that the derivative $f'$ is a function of bounded variation, its distributional derivative $f''$ becomes a well-defined finite Radon measure, which allows us to utilize the Lebesgue decomposition theorem to formally isolate the absolutely continuous and singular components of $f''$.

Our classification relies on the following foundational transversality result by L. Venieri, which guarantees that any dimensional collapse must come from a covering family aligning almost everywhere with the tangent directions.

\begin{theorem}[L. Venieri, 2017  \cite{Venieri2017}]\label{Venieri}
Let $\Omega$ be a domain in $\mathbb{R}^{n-1}$. Let $f\in C(\Omega)$ be Lipschitz continuous and let $A$ be
a subset of the graph $M$ of $f$ with positive $(n-1)$-dimensional Hausdorff measure. Let
$B$ be a subset of $\mathbb{R}^n$ satisfying the following two properties:
\begin{enumerate}
    \item [1.] for every $\mathbf{x}\in A$, the set $B$ contains a line $L_\mathbf{x}$ through $\mathbf{x}$;
    \item [2.] if $M$ has an approximate tangent hyperplane $T_\mathbf{x}M$ at $\mathbf{x}$, then $L_\mathbf{x}$ is not contained in $T_\mathbf{x}M$.
\end{enumerate} 
Then, the Hausdorﬀ
dimension of $B$ is at least $(n+2)/2$.
\end{theorem}

By synthesizing Venieri's transversality constraint with the area formula for Lipschitz mappings, we now establish the following measure theoretic criterion.

\begin{theorem}\label{BV_Dichotomy}
Let $f \in C^1([0,1])$ such that $f'$ is non-constant of bounded variation. Suppose that there exists a measurable subset $\Lambda$ of $\mathbb{R}^2$ such that the set $$E:=\bigcup_{(a,b)\in \Lambda}\{(x,y)\in\mathbb{R}^2\mid y=ax+b\}$$ contains the graph of $f$ and $\dim_\mathcal{H}(E)<2$. Then, the distributional derivative $f''$ is a singular measure.
\end{theorem}
\begin{proof}
Assume first that there exists such a measurable subset $\Lambda$ with $\dim_\mathcal{H}(E) < 2$. Since $E$ covers the graph of $f$ and $\dim_\mathcal{H}(E) < 2$, Theorem \ref{Venieri} implies that the lines in $E$ cannot intersect the graph $\Gamma(f)$ of $f$ transversally on any set of positive Lebesgue measure. Therefore, there exists a subset $A$ of $[0,1]$ of full Lebesgue measure such that for every $x \in A$, the line in $E$ passing through $(x,f(x))$ is precisely the tangent line to $\Gamma(f)$ at $(x,f(x))$.

Now, assume the contrary that the distributional derivative $f''$ is not a singular measure. Then, by the Lebesgue decomposition theorem, there exist a singular measure $\nu$ and a non-zero absolutely continuous measure $\mu$ such that $f''=\nu+\mu$, which in particular implies there exists a subset of $[0,1]$ of positive Lebesgue measure on which $f$ is twice differentiable and $f''(x) \neq 0$. Because $A$ has full Lebesgue measure, the intersection $B = A \cap \{x \in [0,1] : f''(x) \neq 0\}$ also admits positive Lebesgue measure.

Since $f$ is twice differentiable at all points in $B$, there exists a measurable subset $C$ of $B$ of positive Lebesgue measure on which the derivative $f'$ is Lipschitz. Consider the continuous mapping $\Phi\colon [0,1]\times\mathbb{R}\rightarrow \mathbb{R}^2$ defined by
$$
\Phi(x,t) = (t, f(x)+f^{\prime}(x)(t-x)).
$$
The restriction of $\Phi$ to $C \times \mathbb{R}$ is locally Lipschitz. Furthermore, $\Phi$ is differentiable at almost every point in $C \times \mathbb{R}$ with Jacobian determinant $f''(x)(x-t)$. Since $t-x \geq 1 > 0$ for all $x \in C$ and $t\geq2$, on the product set $C \times [2,3]$, the absolute value of the Jacobian determinant of $\Phi$ is strictly positive almost everywhere.

By the area formula for Lipschitz mappings, the 2-dimensional Lebesgue measure of the set $\Phi(C \times [2,3])$ is strictly positive, since
$$
\int_{C \times [2,3]}|f''(x)|(t-x) \,dx\,dt > 0.
$$
Notice that $\Phi(C \times [2,3])$ is a subset of $E$. Consequently, we obtain that $\dim_\mathcal{H}(E) = 2$, which contradicts our initial assumption that $\dim_\mathcal{H}(E)$ is less than $2$. Therefore, $f''$ must be a singular measure. 
\end{proof}

\section{Main Theorem on Polyhedrons}

We have already explored the dimensional rigidity of line families covering planar curves. One may ask how these geometric constraints manifest in higher-dimensional ambient spaces, particularly in $\mathbb{R}^3$.

The study of unions of lines in $\mathbb{R}^3$ is inextricably linked to the notorious Kakeya set conjecture, which asserts that any compact set containing a unit line segment in every direction must possess a Hausdorff and Minkowski dimension of exactly 3. While the two-dimensional case was resolved by Davies in his seminal work \cite{Davies1971}, the three-dimensional case stood as one of the most prominent open problems in geometric measure theory for decades. See \cite{KatzTao2002} and \cite{Wolff1999} for an introduction to the Kakeya conjecture and a survey of historical
progress on the problem.

Recently, in a monumental breakthrough, Wang and Zahl definitively resolved the Kakeya set conjecture in three dimensions. By utilizing a refined multi-scale decomposition and proving that unions of tubes satisfying specific conditions must attain almost maximal volume, they closed the dimensional gap completely. For our purposes, their result provides the ultimate geometric rigidity required for line families whose directions do not collapse into a null set. We state this continuous formulation of their theorem as follows:

\begin{theorem}[Wang-Zahl, 2025]\label{Wang-Zahl}
Let $\Omega$ be a measurable subset of the unit $2$-sphere $$\mathbb{S}^2:=\{\mathbf{x}\in\mathbb{R}^3:\|\mathbf{x}\|=1\},$$ and $F\colon\Omega\rightarrow\mathbb{R}^3$ a measurable vector field. If $\Omega$ enjoys positive Haar measure, then the set 
$$
\bigcup_{\mathbf{x}\in\Omega}\{F(\mathbf{x})+\mathbf{x}t:t\in[0,1]\}
$$
has Hausdorff dimension $3$.
\end{theorem}

With this profound dimensional bound established for families of lines possessing a positive measure of directions, we shift our focus to vector fields defined on the two-dimensional boundaries of three-dimensional spatial domains. Specifically, we examine the case where the domain is a polyhedron. The geometric flatness of polyhedral faces provides a rigid algebraic structure that interacts remarkably well with fluid flux. If an absolutely integrable vector field exhibits a strictly positive total flux across the polyhedral boundary, it must inherently point transversally outward on a subset of a flat face with strictly positive two-dimensional Lebesgue measure. 

By exploiting projective duality, we can map this spatial plane to the plane at infinity. This projective linear transformation elegantly converts the positive measure of spatial base points directly into a positive measure of directions, thereby seamlessly satisfying the exact hypotheses of the Wang-Zahl theorem on Kakeya sets. This enables us to establish our main geometric constraint for polyhedral boundaries:

\begin{theorem}
Let $P$ be a polyhedron in $\mathbb{R}^3$, and $v\in L^1(\partial P)^{\oplus3}$ an absolutely integrable vector field. Suppose that the total flux of $v$ through $\partial P$ is positive. Then the set
$$
\bigcup_{\mathbf{x}\in\partial P}\{\mathbf{x}+v(\mathbf{x})t:t\geq0\}
$$
has Hausdorff dimension $3$.
\end{theorem}
\begin{proof}
By Dirichlet's Schubfachprinzip and the assumption on the total flux, it suffices to prove that, for a measurable subset $\Omega$ of $\mathbb{R}^2$ and measurable functions $f,g\colon\Omega\rightarrow\mathbb{R}$, the set  
$$
E:=\bigcup_{(x,y)\in\Omega}\{(x+f(x,y)t,y+g(x,y)t,t):t>0\}
$$
has Hausdorff dimension $3$ whenever $\Omega$ is not a null set.

Consider the projective linear transformation $$\varphi\colon\{(x,y,z)\in\mathbb{R}^3:z\neq0\}\rightarrow\\\{(x,y,z)\in\mathbb{R}^3:z\neq0\}$$ defined by
$$
\varphi(x,y,z) := \left(\frac{x}{z}, \frac{y}{z}, \frac{1}{z}\right).
$$
Take any $(x,y)\in\Omega$. Applying $\varphi$ to the ray $$L:=\{(x+f(x,y)t,y+g(x,y)t,t):t>0\}$$ yields that
\begin{align*}
    \varphi(L) =& \{(f(x,y)+x/t,g(x,y)+y/t,1/t):t>0\}\\=&\{(f(x,y)+xt/\rho,g(x,y)+yt/\rho,t/\rho):t>0\}
\end{align*}
for $\rho(x,y):=\sqrt{x^2+y^2+1}$. 

Since
$
\{ (x,y,1)/\rho(x,y): (x,y)\in\Omega \}
$
is a measurable set of positive measure, as $\Omega$ is, the set $\varphi(E)$ has Hausdorff dimension $3$ by Theorem \ref{Wang-Zahl}. Finally, because $\varphi$ is a smooth involution, we conclude that $\dim_\mathcal{H}(E)=3$.
\end{proof}

\appendix

\section{Methods from Topology}

In this section, we develop a purely topological framework to analyze the geometric rigidity of continuous line families. While our previous results relied heavily on measure theoretic tools such as the area formula and projection theorems, we demonstrate here that the dimensional collapse of the union of lines can also be understood through the lens of algebraic topology. 

We begin by establishing a strengthened variation of the open mapping theorem. Specifically, we prove that a continuous map possessing a full rank differential at merely a single point is forced to map onto a set with a non-empty interior.

\begin{lemma}\label{Brouwer}
    Let $0 \leq n \leq m$, and let $f: \mathbb{R}^{m} \rightarrow \mathbb{R}^{n}$ be a continuous map differentiable at $\mathbf{0} \in \mathbb{R}^{m}$ such that the Jacobian matrix $J$ of $f$ at $\mathbf{0}$ is of full rank, so in particular $J$ admits a left inverse $J^{-1}$ with $$\lambda:=\dfrac{1}{2\|J^{-1}\|_\mathrm{o p}}>0.$$ Then, there exists $ \varepsilon>0$ such that $$f^{-1}(\mathbf{v}) \cap \{\mathbf{x}\in\mathbb{R}^{m}:\|\mathbf{x}\|\leq\varepsilon\} \neq \varnothing$$ and that $$ \lambda\|\mathbf{u}\| \leq\|\mathbf{v}-f(\mathbf{0})\| \leq(\|J\|_\mathrm{o p}+\lambda)\|\mathbf{u}\|$$ for all $\mathbf{v} \in \{\mathbf{y}\in\mathbb{R}^{n}:\|\mathbf{y}-f(\mathbf{0})\|\leq\lambda\varepsilon\}$ and $ \mathbf{u} \in f^{-1}(\mathbf{v}) \cap \{\mathbf{x}\in\mathbb{R}^{m}:\|\mathbf{x}\|\leq\varepsilon\}$.
\end{lemma}
\begin{proof}
    By the definition of derivatives, there exists sufficiently small $\varepsilon>0$ such that $$\|(f(\mathbf{u})-f(\mathbf{0}))-J\mathbf{u}\|<\lambda\|\mathbf{u}\|$$ for all $\mathbf{u} \in \{\mathbf{x}\in\mathbb{R}^{m}:\|\mathbf{x}\|\leq\varepsilon\}$. For any $\mathbf{v} \in \mathbb{R}^{n}$, we define $\Psi_{\mathbf{v}}: \mathbb{R}^{m} \rightarrow \mathbb{R}^{n}$ by $$\Psi_{\mathbf{v}}(\mathbf{u}):=\mathbf{u}-J^{-1}(f(\mathbf{u})-\mathbf{v}).$$ Then, $\Psi_{\mathbf{v}}$ is continuous, and $\Psi_{\mathbf{v}}(\mathbf{u})=\mathbf{u}$ if and only if $J^{-1}(f(\mathbf{u})-\mathbf{v})=0$, that is, $\mathbf{v}=f(\mathbf{u})$. Moreover, we have 
    $$
    \begin{aligned}
    -\Psi_{\mathbf{v}}(\mathbf{u}) &=\mathbf{0}-\mathbf{u}+J^{-1}(f(\mathbf{u})-\mathbf{v}) \\
    &=(\mathbf{0}-\mathbf{u})+J^{-1}(f(\mathbf{u})-f(\mathbf{0}))+J^{-1}(f(\mathbf{0})-\mathbf{v}) \\
    &=J^{-1}((f(\mathbf{u})-f(\mathbf{0}))-J\mathbf{u})+J^{-1}(f(\mathbf{0})-\mathbf{v})
    \end{aligned}
    $$
    Fix $\mathbf{v} \in \{\mathbf{y}\in\mathbb{R}^{n}:\|\mathbf{y}-f(\mathbf{0})\|\leq\lambda\varepsilon\}$. For any $\mathbf{u} \in \{\mathbf{x}\in\mathbb{R}^{m}:\|\mathbf{x}\|\leq\varepsilon\}$, we have
    \begin{equation}\label{eq1}\tag{*}
        \begin{aligned}
        \|\Psi_{\mathbf{v}}(\mathbf{u})\| & \leq\|J^{-1}\|_\mathrm{o p}\|(f(\mathbf{u})-f(\mathbf{0}))-J\mathbf{u}\|+\|J^{-1}\|_\mathrm{o p}\|\mathbf{v}-f(\mathbf{0})\| \\
        & \leq \dfrac{1}{2 \lambda} \cdot \lambda\|\mathbf{u}\|+\dfrac{1}{2 \lambda}\|\mathbf{v}-f(\mathbf{0})\|\\ 
        & \leq \dfrac{1}{2}(\|\mathbf{u}\|+\lambda^{-1}\|\mathbf{v}-f(\mathbf{0})\|)\\ 
        &\leq \dfrac{1}{2}(\varepsilon+\lambda^{-1}\cdot \lambda \varepsilon)\\ 
        & =\varepsilon
        \end{aligned}
        \end{equation}
        and hence $\Psi_{\mathbf{v}}(\mathbf{u}) \in \{\mathbf{x}\in\mathbb{R}^{m}:\|\mathbf{x}\|\leq\varepsilon\}$. Since $\Psi_{\mathbf{v}}(\mathbf{u})$ is continuous on $\{\mathbf{x}\in\mathbb{R}^{m}:\|\mathbf{x}\|\leq\varepsilon\}$ and $$\Psi_{\mathbf{v}}(\{\mathbf{x}\in\mathbb{R}^{m}:\|\mathbf{x}\|\leq\varepsilon\}) \subseteq \{\mathbf{x}\in\mathbb{R}^{m}:\|\mathbf{x}\|\leq\varepsilon\},$$ by Brouwer's fixed point theorem, we conclude that there exists $\mathbf{u}_{*} \in \{\mathbf{x}\in\mathbb{R}^{m}:\|\mathbf{x}\|\leq\varepsilon\}$ such that $\Psi_{\mathbf{v}}(\mathbf{u}_{*})=\mathbf{u}_{*}$, which implies $\mathbf{v}=f(\mathbf{u}_{*})$. Therefore, we have $$\mathbf{u}_{*} \in A:=f^{-1}(\mathbf{v}) \cap \{\mathbf{x}\in\mathbb{R}^{m}:\|\mathbf{x}\|\leq\varepsilon\}$$ and hence $A$ is non-empty. 
        
        Take any $\mathbf{u} \in A$. Then $\mathbf{v}=f(\mathbf{u})$ and hence $\Psi_{\mathbf{v}}(\mathbf{u})=\mathbf{u}$. Therefore by the estimate (\ref{eq1}), we immediately have $$\|\mathbf{u}\|=\|\mathbf{0}-\Psi_{\mathbf{v}}(\mathbf{u})\| \leq \dfrac{1}{2}(\|\mathbf{u}\|+\lambda^{-1}\|\mathbf{v}-f(\mathbf{0})\|),$$ and hence $\lambda\|\mathbf{u}\| \leq\|\mathbf{v}-f(\mathbf{0})\|$. But, we also have
        $$
        \begin{aligned}
        \|\mathbf{v}-f(\mathbf{0})\| &=\|f(\mathbf{u})-f(\mathbf{0})\| \\
        & \leq\|f(\mathbf{u})-f(\mathbf{0})-J\mathbf{u}\|+\|J\mathbf{u}\| \\
        & \leq \lambda\|\mathbf{u}\|+\|J\|_\mathrm{o p}\|\mathbf{u}\| \\
        &=(\|J\|_\mathrm{o p}+\lambda)\|\mathbf{u}\|.
        \end{aligned}
        $$
        Therefore, we finally arrive at the desired inequality $\lambda\|\mathbf{u}\| \leq\|\mathbf{v}-f(\mathbf{0})\| \leq(\|J\|_\mathrm{o p}+\lambda)\|\mathbf{u}\|$. 
\end{proof}

\begin{corollary}\label{Brouwer_cor}
    Let $0 \leq n \leq m$ and let $f: \mathbb{R}^m \rightarrow \mathbb{R}^{n}$ be a continuous map. If $f$ is differentiable at $\mathbf{0}\in\mathbb{R}^m$ and the Jacobian matrix of $f$ at $\mathbf{0}$ is of full rank, then the interior of $f(\mathbb{R}^{m})$ is non-empty.
\end{corollary}
\begin{proof}
    By Lemma \ref{Brouwer}, there exists $\delta>0$ such that for all $\mathbf{v} \in \{\mathbf{y}\in\mathbb{R}^{n}:\|\mathbf{y}-f(\mathbf{0})\|\leq\delta\}$ the fiber $f^{-1}(\mathbf{v})$ over $\mathbf{v}$ is non-empty. Therefore $f(\mathbb{R}^{m})$ admits a non-empty interior.
\end{proof}

With this topological foundation established, we can now study the geometry of $1$-parameter families of lines. The following proposition proves that if a continuous family of lines collapses into a set with an empty interior, its parameterizing function is forced to be nowhere differentiable.

\begin{proposition}
Let $f\in C([0,1])$. Suppose that the interior of the set 
$$\bigcup_{p\in[0,1]}\{(x,y)\in\mathbb{R}^2:y=px+f(p)\}$$ is empty, then $f$ is nowhere differentiable.
\end{proposition}
\begin{proof}
Assume for the sake of contradiction that $f$ is differentiable at some point $a \in [0,1]$. Consider the continuous mapping $\Phi\colon [0,1] \times \mathbb{R} \to \mathbb{R}^2$ defined by
$$
\Phi(p,x) := (x, px + f(p)).
$$
The image of $\Phi$ is precisely $$E := \bigcup_{p\in[0,1]}\{(x,y)\in\mathbb{R}^2 \mid y=px+f(p)\}.$$ Because $f$ is differentiable at $a$, the mapping $\Phi$ is differentiable at $(a, x)$ for any spatial coordinate $x \in \mathbb{R}$. The Jacobian matrix of $\Phi$ at the point $(a, 1-f'(a))$ is
$$
J= \begin{bmatrix} 0 & 1 \\ 1-f'(a) + f'(a) & a \end{bmatrix}=\begin{bmatrix} 0 & 1 \\ 1 & a \end{bmatrix}.
$$
Since $\det(J)=-1$, the matrix $J$ achieves full rank.

Now, by Corollary \ref{Brouwer_cor}, we conclude that $E$ contains an open disk. 
\end{proof}
\section{Generalization to Hypersurfaces}

In this appendix, we demonstrate that the high-dimensional analogue of the geometric rigidity established for planar curves holds true. Specifically, we prove that the union of tangent hyperplanes to a generic differentiable hypersurface has full Hausdorff dimension.

\begin{definition}
Let $\Omega$ be a domain in $\mathbb{R}^d$ and let $f\colon \Omega \rightarrow \mathbb{R}$ be an arbitrary function. $f$ is termed a linear function if there exists a constant vector $\mathbf{a} \in \mathbb{R}^d$ and a scalar $b \in \mathbb{R}$ such that $f(\mathbf{x})=\langle\mathbf{a}, \mathbf{x}\rangle+b$ for all $\mathbf{x} \in \Omega$.
\end{definition}

Recall from calculus that a differentiable function $f$ is linear if and only if its gradient $\nabla f$ is constant on $\Omega$.

\begin{lemma}\label{lemmaB.2}
Let $\left\{f_n\right\}_{n=1}^{\infty}$ be a sequence of linear functions defined on an interval $I$ in $\mathbb{R}$. If $\left\{f_n\right\}_{n=1}^{\infty}$ converges pointwise, then its limit is a linear function.
\end{lemma}
\begin{proof}
By assumption, for each $n\geq1$ there exist real numbers $a_n$ and $b_n$ such that $$f_n(t)=a_n t+b_n$$ for all $t \in I$. Fix two distinct points $x,y\in I$. The sequences of evaluations $f_n(x)$ and $f_n(y)$ converge. Consequently, the sequences of slopes $$a_n = \frac{f_n(x)-f_n(y)}{x-y}$$ and intercepts $b_n$ must also converge to some limits $a$ and $b$, respectively. Thus, for all $t \in I$, we have $$\lim_{n \rightarrow \infty} f_n(t) = a t+b,$$ that is, the sequence $\left\{f_n\right\}_{n=1}^{\infty}$  converges to the linear function $f(t)=at+b$.
\end{proof}

\begin{lemma}\label{lemmaB.3}
Let $\Omega$ be a domain in $\mathbb{R}^d$ and $f\colon \Omega \rightarrow \mathbb{R}$ be a differentiable function. If $f$ is not a linear function, then there exists a point $\mathbf{x} \in \Omega$ such that $\left.f\right|_U$ is not a linear function for any neighborhood $U$ of $\mathbf{x}$ in $\Omega$.
\end{lemma}
\begin{proof}
Assume the contrary that for each $\mathbf{x} \in \Omega$ there exists a neighborhood $U$ of $\mathbf{x}$ such that $\left.f\right|_U$ is a linear function. This implies that the gradient $\nabla f$ is constant on $U$, meaning $\nabla f$ is locally constant everywhere in $\Omega$. Since $\Omega$ is a connected domain, a locally constant function must be globally constant. Thus $\nabla f$ is globally constant, which implies $f$ is a linear function, yielding a contradiction.
\end{proof}

To bound the dimension of the resulting union of hyperplanes, we must guarantee that a nonlinear function contains a sufficiently large family of nonlinear 1-dimensional slices. It is a standard fact of geometry that if a continuous function is linear on every possible line segment within its domain, then the function is globally linear. We utilize this to establish the following slice condition:

\begin{lemma}\label{lemmaB.4}
Let $\Omega$ be a domain in $\mathbb{R}^d$ and $f\in C(\Omega)$. If $f$ is not a linear function, then there exists a unit vector $\mathbf{v} \in \mathbb{R}^d$ and an open ball $B$ in the orthogonal complement $\operatorname{Span}(\mathbf{v})^\perp$ such that, for all points $\mathbf{a} \in B$, the function $f_{\mathbf{a}}(t) := f(\mathbf{a} + t\mathbf{v})$ is not linear.
\end{lemma}
\begin{proof}
Assume the contrary. Then, for every unit vector $\mathbf{v} \in \mathbb{R}^d$ and any point $\mathbf{a} \in \operatorname{Span}(\mathbf{v})^\perp$, there exists a sequence $\{\mathbf{a}_i\}_{i=1}^\infty$ in $\mathbb{R}^d$ converging to $ \mathbf{a}$ such that $f_{\mathbf{a}_i}$ is linear for every $i\geq 1$. By the continuity of $f$, the sequence $f_{\mathbf{a}_i}$ converges to $f_{\mathbf{a}}$ pointwise. By Lemma \ref{lemmaB.2}, the pointwise limit $f_{\mathbf{a}}$ must also be a linear function. 

This forces both $f$ and $-f$ to be convex, which then implies that $f$ itself is linear, contradicting our initial assumption. The proof is therefore completed.
\end{proof}

\begin{theorem}[Fubini-type inequality]\label{Fubini-type inequality}
Let $0\leq\alpha\leq n$ and $0\leq\beta\leq m$, where $n, m \in \mathbb{Z}$ are positive integers. Let $E$ be a subset of $\mathbb{R}^n \times \mathbb{R}^m$. If there exists a subset $A$ of $\mathbb{R}^n$ such that $\dim_{\mathcal{H}}(A) \geq \alpha$ and that $$\dim_{\mathcal{H}}(\{\mathbf{y} \in \mathbb{R}^m : (\mathbf{a}, \mathbf{y}) \in E\}) \geq \beta$$ for all $\mathbf{a} \in A$, then $\dim_{\mathcal{H}}(E) \geq \alpha+\beta$.
\end{theorem}
\begin{proof}
This is a classical result in geometric measure theory regarding the Hausdorff dimension of product spaces, proved, for example, by Falconer in \cite{Falconer1985}.
\end{proof}

\begin{corollary}\label{theoremA}
Let $\Omega$ be a domain in $\mathbb{R}^d$, and let $f\colon \Omega \rightarrow \mathbb{R}$ be a nonlinear differentiable function. Then the Hausdorff dimension of
$$
\bigcup_{\mathbf{a} \in \Omega}\{(\mathbf{x}, y) \in \Omega \times \mathbb{R} \mid y=f(\mathbf{a})+\langle\nabla f(\mathbf{a}), \mathbf{x}-\mathbf{a}\rangle\}
$$
is precisely $d+1$.
\end{corollary}
\begin{proof}
By Lemma \ref{lemmaB.3}, there exists a point $\mathbf{b} \in \Omega$ such that $\left.f\right|_U$ is not a linear function for any neighborhood $U$ of $\mathbf{b}$ in $\Omega$. By Lemma \ref{lemmaB.4}, there exist a unit vector $\mathbf{v}$ and an open ball $B$ in the orthogonal complement $\operatorname{Span}(\mathbf{v})^\perp$ such that the function $$f_{\mathbf{a}}\colon I \to \mathbb{R}, \quad t\mapsto f(\mathbf{a} + t\mathbf{v})$$ is not linear for all $\mathbf{a} \in B$ and some sufficiently small interval $I$. 

For any fixed $\mathbf{a} \in B$, because $f$ is differentiable, the restriction $f_{\mathbf{a}}$ is a nonlinear differentiable function of one variable. By Corollary \ref{differentiable3}, the union of the tangent lines to the graph of $f_{\mathbf{a}}$ is a set of Hausdorff dimension exactly $\beta=2$.

Notice further that $B$ is an open ball with Hausdorff dimension $\alpha = d-1$. The Fubini-type inequality in Theorem \ref{Fubini-type inequality} immediately yields that the Hausdorff dimension of 
$$
E:=\bigcup_{\mathbf{a} \in \Omega}\{(\mathbf{x}, y) \in \Omega \times \mathbb{R} \mid y=f(\mathbf{a})+\langle\nabla f(\mathbf{a}), \mathbf{x}-\mathbf{a}\rangle\}
$$
is at least $\alpha+\beta=(d-1)+2=d+1$.
However, because $E$ is a subset of the ambient space $\mathbb{R}^{d+1}$, its dimension cannot exceed $d+1$. We therefore conclude that $\dim_\mathcal{H}(E)$ is exactly $d+1$.
\end{proof}

\section*{Acknowledgement}

The author is deeply grateful to Tam\'as Keleti for drawing the author's attention to this fascinating topic and for inspiring this line of inquiry. Also, the author would like to express sincere gratitude to Weiyi Zhang for very useful and inspiring discussions.

Furthermore, the author is also grateful to
Tam\'as Keleti for pointing out a mathematical inaccuracy that occurred in Section 4, and for suggesting an idea on how to fix that.

This research was completed while the author was studying at the Mathematics Institute of the University of Warwick. The author would like to thank the University of Warwick for its hospitality.



\section*{Statements and Declarations}

No funding was received to assist with the preparation of this manuscript, and the author did not receive support from any organization for the submitted work.

The author certifies that they have no affiliations with or involvement in any other organization or entity with any financial interest or non-financial interest in the subject matter or materials discussed in this manuscript.

This article is licensed under a Creative Commons Attribution 4.0 International License, which permits use, sharing, adaptation, distribution and reproduction in
any medium or format, as long as you give appropriate credit to the original author and
the source, provide a link to the Creative Commons licence, and indicate if changes were
made. The images or other third party material in this article are included in the article’s
Creative Commons licence, unless indicated otherwise in a credit line to the material. If
material is not included in the article’s Creative Commons licence and your intended use is
not permitted by statutory regulation or exceeds the permitted use, you will need to obtain
permission directly from the copyright holder.

Data sharing is not applicable to this article as no datasets were generated or analysed during the current study.

The author hereby provides consent for the publication of the manuscript detailed above.

\bibliographystyle{plain}
\bibliography{references}

@article{Davies1952,
  author    = "Davies, R. O.",
  title     = "On accessibility of plane sets and differentiation of functions of two variables",
  journal   = "Mathematical Proceedings of the Cambridge Philosophical Society",
  volume    = "48",
  pages     = "215--232",
  year      = "1952"
}

@article{Davies1971,
  author    = "Davies, R. O.",
  title     = "Some remarks on the {K}akeya problem",
  journal   = "Mathematical Proceedings of the Cambridge Philosophical Society",
  volume    = "69",
  pages     = "417--421",
  year      = "1971"
}

@article{KatzTao2002,
  author    = "Katz, N. and Tao, T.",
  title     = "New bounds for {K}akeya problems",
  journal   = "Journal d'Analyse Math{\'e}matique",
  volume    = "87",
  pages     = "231--263",
  year      = "2002"
}

@article{KatzZahl2019,
  author    = "Katz, N. and Zahl, J.",
  title     = "An improved bound on the {H}ausdorff dimension of {B}esicovitch sets in {R}3",
  journal   = "Journal of the American Mathematical Society",
  volume    = "32",
  pages     = "195--259",
  year      = "2019"
}

@article{Venieri2017,
  author    = "Venieri, L.",
  title     = "Dimension estimates for {K}akeya sets defined in an axiomatic setting",
  journal   = "Annales Academi{\ae} Scientiarum Fennic{\ae} Mathematica Dissertationes",
  volume    = "161",
  year      = "2017"
}

@article{Cumberbatch2023,
  author    = "Cumberbatch, James and Keleti, Tam{\'a}s and Zhang, Jialin",
  title     = "Hausdorff dimension of union of lines that cover a curve",
  journal   = "Pure and Applied Functional Analysis",
  volume    = "8",
  pages     = "1651--1659",
  year      = "2023"
}

@article{Falconer2016,
  author    = "Falconer, K. and Mattila, P.",
  title     = "Strong {M}arstrand theorems and dimensions of sets formed by subsets of hyperplanes",
  journal   = "Journal of Fractal Geometry",
  volume    = "3",
  pages     = "319--326",
  year      = "2016"
}

@book{Kechris1995,
  author    = "Alexander S. Kechris",
  title     = "Classical descriptive set theory",
  publisher = "Springer New York, NY",
  year      = "1995"
}

@book{Falconer1985,
  author    = "K. Falconer",
  title     = "The Geometry of Fractal Sets",
  publisher = "Cambridge University Press",
  year      = "1985"
}

@incollection{Wolff1999,
  author    = "Wolff, T.",
  title     = "Recent work connected with the {K}akeya problem",
  booktitle = "Prospects in mathematics ({P}rinceton, {NJ}, 1996)",
  publisher = "Amer. Math. Soc.",
  address   = "Providence, {RI}",
  pages     = "129--162",
  year      = "1999"
}

\end{document}